\documentclass[12pt,english]{amsart}
\usepackage[cp1251]{inputenc}
\usepackage[english]{babel}
\usepackage{amsmath,amsthm,amssymb}
\newtheorem{theorem}{Theorem}[section]
\newtheorem{lemma}[theorem]{Lemma}
\newtheorem{corollary}[theorem]{Corollary}
\newtheorem{remark}[theorem]{Remark}

\numberwithin{equation}{section}

\begin{document}

\title[Deviation of polynomials from their expectations and isoperimetry]
{Deviation of polynomials from their expectations and isoperimetry}

\author[L.M. Arutyunyan]{Lavrentin M. Arutyunyan}
\address{Faculty of Mechanics and Mathematics,
Moscow State University, Moscow, 119991 Russia}
\curraddr{}
\email{lavrentin@yandex.ru}
\thanks{This work has been supported by the Russian Science Foundation Grant 14-11-00196 at
Lomonosov Moscow State University.}

\author[E.D. Kosov]{Egor~D.~Kosov}
\address{Faculty of Mechanics and Mathematics,
Moscow State University, Moscow, 119991 Russia}
\curraddr{}
\email{ked\_2006@mail.ru}
\thanks{}



\date{}

\begin{abstract}
In the first part we study deviation of a polynomial from its mathematical expectation. This deviation
can be estimated from above by Carbery--Wright inequality, so we investigate estimates of the deviation from below.
We obtain such estimates in two different cases: for Gaussian measures and a polynomial of an arbitrary degree and
for an arbitrary log-concave measure but only for polynomials of the second degree.
In the second part we deals with isoperimetric inequality and the Poincar\'e inequality
for probability measures on the real line
that are images of the uniform distributions on convex compact sets in
$\mathbb{R}^n$ under polynomial mappings.
\end{abstract}

\maketitle

\vskip .1in

Polynomials on spaces with log-concave measures possess a number of important and useful properties.
These properties have been studied in many works, see for example \cite{ArutKos, BobkPoly, Bobk, Bourg, NSV}.
Some authors (see \cite{BKNP, BZ, NNP, NP}) also consider the special case of Gaussian measures.

Results about polynomial distributions find many applications in various fields.
One of such applications is concerned with
geometrical properties of convex bodies, especially
when the dimension tends to infinity. For example, Bourgain \cite{Bourg} proved an upper bound
in the hyperplane conjecture
by using a Khinchin-type inequality for polynomials of a fixed degree on convex bodies
(about this inequality see also \cite{BobkPoly, Bobk}). It should be remarked that the best known upper bound
in the hyperplane conjecture is due to Klartag (see \cite{Klartag}).
On the other hand, properties of polynomials and polynomial distributions play an important role
in probabilistic questions. In particular, Gaussian measures are also logarithmically concave
and there are certain properties of Gaussian measures that were proved only in the framework
of general logarithmically concave measures.
One of them is the following Carbery--Wright inequality
that holds for any polynomial $f$ of degree $d$ and any
log-concave measure $\mu$ on $\mathbb{R}^n$ (see \cite{CarWr}):
$$
\|f\|^{1/d}_{L^1(\mu)}\mu(x\colon\ |f(x)|\le\alpha)\le C(d)\alpha^{1/d}.
$$
This inequality has already found interesting applications in probability theory
(see, for example, \cite{NNP, NP}).

In the first part of this work we discuss inequalities that are
reverse in some sense to the Carbery--Wright inequality,
more precisely, inequalities of the form
$$
\mu(|f-m_f|\le \sigma_f s)\ge C(d)\varphi(s),
$$
where $f$ is a polynomial of degree $d$ and $m_f$ and $\sigma_f^2$
are its expectation and variance, respectively,
and $s\in[0, 1/2]$.
We prove inequalities of such a type in two different cases.
In the first case, the measure $\mu$ is Gaussian and $\varphi(s)=s |\ln s|^{-d/2}$
(see Theorem \ref{t2.1}).
In the second case, the measure $\mu$ is an arbitrary log-concave measure and
the degree of the polynomial $f$ is at most two while $\varphi(s)=s$ (see Corollary \ref{c2.2}).
The main feature of our inequalities is their independence of the dimension of the space
and of the measure $\mu$ itself.

The second part of our work is devoted to the isoperimetric inequality and the Poincar\'e inequality
for distributions of polynomials.
It is well-known that for the standard Gaussian measure on $\mathbb{R}^n$
both inequalities hold true
(see \cite{Gaus, STs, Bor3}).
For log-concave measures, inequalities of these types have been
studied in \cite{BobkIsop, KLS}.
In our work, these inequalities are proved for probability measures on the real line
that are polynomial images of the uniform distributions on convex compact sets in
$\mathbb{R}^n$ (see Theorem \ref{t3.1} and Corollaries \ref{c3.2} and \ref{c3.3}).

One of the main tools used in this paper is the
so-called localization technique (see \cite{KLS, LS}).
The idea of localization of a problem
was used in many papers as an approach for obtaining estimates in multidimensional spaces.
For example, it was used to study isoperimetric inequalities for the uniform distributions on
convex bodies in \cite{KLS} and for the distributions on spheres in \cite{GrM1}. Also it was used in
\cite{Bobk, BobkPoly} in the proof of Khinchine-type inequalities for polynomials.
This technique allows to reduce some multidimensional inequalities to one-dimensional ones.
In the case of polynomials it is especially convenient, because a restriction of a polynomial to a
straight line is again a polynomial. A new approach to the ideas of localization
was developed in \cite{FrGue},
where localization is interpreted as a property
of extreme points of some special convex sets in the space of all probability measures.

\section{Preliminaries}

First of all we introduce  some notation and mention ceratin
previously known results used in our work.

Let $\mu$ be a probability Borel measure on $\mathbb{R}^n$ and
let $f$ be a $\mu$-measurable function.
We use the following notation:
$$
\mu_f=\mu\circ f^{-1} \text{ is the image of the measure } \mu \text{ under the mapping } f,
$$
$$
m_f=\int fd\mu \text{ is the expectation of the random variable } f,
$$
$$
\sigma_f^2=\int(f-m_f)^2d\mu \text{ is the variance of the random variable } f,
$$
$$
\alpha_f=\int|f-m_f|d\mu,
$$
$$
\|f\|_{p}=\biggl(\int |f|^p d\mu\biggr)^{1/p} \text{for} \ p>0, \quad \|f\|_0=\exp\biggl(\int \ln|f|d\mu\biggr)=\lim\limits_{r\to 0} \|f\|_r.
$$

Let $I_A$ denote the indicator function of a set $A$.

A probability Borel measure $\mu$ on $\mathbb{R}^n$ is called logarithmically concave
(also log-concave or convex) if
it has a density of the form $e^{-V}$ with respect to Lebesgue measure on some affine subspace,
where $V$ is a convex function (possibly with infinite values) on this subspace (see~\cite{DiffMeas}).
This properety is equivalent (see \cite{Bor, Bor2}) to the property that for every pair
of Borel sets $A,B$ one has
$$
    \mu(t A + (1-t)B) \ge \mu(A)^{t}\mu(B)^{1-t}, \ \forall\ t\in [0,1].
$$

A polynomial of degree $d$ is a function
$f$ on $\mathbb{R}^n$ of the form
$$
    f(x)=\sum_{m=0}^d B_m(x,\ldots,x),
$$
where $B(x_1,\ldots,x_m)$ is a symmetric  $m$-linear function.

Let $\nu$ be a probability Borel measure on the real line.
Define the $\nu$-perimeter of a set $A$ by the following formula:
$$
\nu^+(A)=\liminf_{\varepsilon\to0}\frac{\nu(A+(-\varepsilon, \varepsilon))-\nu(A)}{\varepsilon}.
$$

\vskip .1in

The proofs of our main results use the following known facts.

\begin{theorem}[see \cite{NSV}]\label{t1.1}
Let $\mu$ be a log-concave measure on $\mathbb{R}^n$ and let $f$
be a polynomial of degree $d$. Set $k(f)=\inf\{k\colon\ \mu(\{|f|\ge k\})\le1/e\}$. Then
for every $t\ge1$ one has
$$
    \mu(x\colon\ |f(x)|\ge (4t)^dk(f))\le e^{-t}.
$$
\end{theorem}

\begin{theorem}[see \cite{BobkPoly, Bobk}]\label{t1.2}
Let $\mu$ be a log-concave measure on $\mathbb{R}^n$, $q\ge1$. Then there
is an absolute constant $c$ such that for every
polynomial $f$ of degree $d$ the following inequalities hold true:
$$
\|f\|_q\le (cqd)^d\|f\|_0, \quad \|f\|_q\le (cq)^d\|f\|_1.
$$
\end{theorem}

Some infinite-dimensional analogues of these two theorems are presented in \cite{ArutKos}.

\begin{theorem}[see \cite{ArutKos}]\label{t1.3}
Let $\mu$ be a log-concave measure on $\mathbb{R}^n$ and let $U$ be a set of positive $\mu$-measure.
Then there is an absolute constant $C$ such that, for every polynomial $f$ of degree $d$,
the following estimate holds:
$$
    \mu(U)^{d+1}\int|f|d\mu\le (Cd)^{2d}\int_U|f|d\mu.
$$
\end{theorem}

Let us recall two localization lemmas.
The first one is concerned with log-concave measures.

\begin{theorem}[see \cite{FrGue, KLS, LS}]\label{t1.4}
Let $f_1, f_2$ be a pair of two upper semi-continuous nonnegative functions on $\mathbb{R}^n$ and let
$f_3, f_4$ be a pair of two lower semi-continuous nonnegative functions on $\mathbb{R}^n$.
Suppose that
for every compact interval $\Delta=[a,b]\subset\mathbb{R}^n$ and
every measure $\nu$ with a density of the form $e^\ell$
with respect to Lebesgue measure on $\Delta$,
where $\ell$ is an affine function on $\Delta$, one has
$$
    \biggl(\int_\Delta f_1d\nu\biggr)^\alpha\biggl(\int_\Delta f_2d\nu\biggr)^\beta\le
    \biggl(\int_\Delta f_3d\nu\biggr)^\alpha\biggl(\int_\Delta f_4d\nu\biggr)^\beta.
$$
Then the following inequality holds for every log-concave measure $\mu$ on $\mathbb{R}^n$:
$$
    \biggl(\int f_1d\mu\biggr)^\alpha\biggl(\int f_2d\mu\biggr)^\beta\le
    \biggl(\int f_3d\mu\biggr)^\alpha\biggl(\int f_4d\mu\biggr)^\beta.
$$
\end{theorem}

The second localization lemma is applicable in the case of uniform distributions on convex bodies.

\begin{theorem}[see \cite{FrGue, KLS, LS}]\label{t1.5}
Let $f_1, f_2$ be a pair of two upper semi-continuous nonnegative functions on $\mathbb{R}^n$ and let
$f_3, f_4$ be a pair of two lower semi-continuous nonnegative functions on $\mathbb{R}^n$.
Suppose that
for every compact interval $\Delta=[a,b]\subset\mathbb{R}^n$ and every measure $\nu$
with a density of the form $(\alpha t + \beta)^{n-1}$
with respect to Lebesgue measure on $\Delta$ one has
$$
    \biggl(\int_\Delta f_1d\nu\biggr)^\alpha\biggl(\int_\Delta f_2d\nu\biggr)^\beta\le
    \biggl(\int_\Delta f_3d\nu\biggr)^\alpha\biggl(\int_\Delta f_4d\nu\biggr)^\beta.
$$
Then the following inequality holds for every convex body $K$ in $\mathbb{R}^n$:
$$
    \biggl(\int_K f_1d\lambda\biggr)^\alpha\biggl(\int_K f_2d\lambda\biggr)^\beta\le
    \biggl(\int_K f_3d\lambda\biggr)^\alpha\biggl(\int_K f_4d\lambda\biggr)^\beta,
$$
where $\lambda$ is Lebesgue measure.
\end{theorem}

\section{Behavior of the  distribution
of a polynomial in a neighbourhood of its expectation}

In this section we estimate from below the measure of small deviations of a polynomial from its mean.
Firstly, we need to establish the following important property of the expectation of a polynomial.

\begin{lemma}\label{lem2.1}
Let $\mu$ be a log-concave measure on $\mathbb{R}^n$.
Then there are positive constants $c(d)$ and $s(d)$ depending only on the degree $d$ such that
for every polynomial $f$
of degree $d$ with $\alpha_f>0$ and $m_f=0$ the following estimate holds true:
$$
\mu(f\ge\varepsilon\alpha_f)\ge c(d) \quad \text{for every number}\ \varepsilon\in [0, s(d)].
$$
\end{lemma}
\begin{proof}
Let $\delta=\mu(f\ge\varepsilon\alpha_f)$.
Applying the inequality from Theorem \ref{t1.3} to the set $U=\{f<\varepsilon \alpha_f\}$ we get
\begin{multline*}
(1-\delta)^{d+1}\int_{\mathbb{R}^n}|f|d\mu=\mu(f<\varepsilon\alpha_f)^{d+1}
\int_{\mathbb{R}^n}|f|d\mu\\
\le(Cd)^{2d}\int_{\{f<\varepsilon\alpha_f\}}|f|d\mu=
(Cd)^{2d}\biggl(\int_{\{0<f<\varepsilon\alpha_f\}}|f|d\mu+\int_{\{f<0\}}|f|d\mu\biggr)\\=
(Cd)^{2d}\biggl(\int_{\{0<f<\varepsilon\alpha_f\}}|f|d\mu+\int_{\{f>0\}}|f|d\mu\biggr)\\=
(Cd)^{2d}\biggl(2\int_{\{0<f<\varepsilon\alpha_f\}}|f|d\mu+\int_{\{f\ge\varepsilon\alpha_f\}}|f|d\mu\biggr).
\end{multline*}
We now estimate $\displaystyle \int_{\{f\ge\varepsilon\alpha_f\}}|f|d\mu$ from above.
For this purpose we use Theorem \ref{t1.1}. Let $\tau>1$
(this number will be chosen later). Then
\begin{multline*}
\int_{\{f\ge\varepsilon\alpha_f\}}|f|d\mu\\
\le\int_{\{4^dk(f)\tau>f\ge\varepsilon\alpha_f\}}|f|d\mu+\int_{\{|f|\ge4^dk(f)\tau\}}|f|d\mu \le
4^dk(f)\tau\delta+\int_{4^dk(f)\tau}^\infty\mu\{|f|>t\}dt\\=
4^dk(f)\tau\delta+d4^dk(f)
\int_{\tau^{1/d}}^\infty\lambda^{d-1}\mu\{|f|>(4\lambda)^dk(f)\}d\lambda\\
\le4^dk(f)\tau\delta+d4^dk(f)\int_{\tau^{1/d}}^\infty\lambda^{d-1}e^{-\lambda}d\lambda.
\end{multline*}
Since $\displaystyle k(f)\le e\int_E|f|d\mu$, the last expression can be estimated from above by
$$
4^de\int_{\mathbb{R}^n}|f|d\mu(\tau\delta+d\int_{\tau^{1/d}}^\infty\lambda^{d-1}e^{-\lambda}d\lambda).
$$
Thus, using the estimate
$$
\int_{\{0<f<\varepsilon\alpha_f\}}|f|d\mu\le s(d)\int_{\mathbb{R}^n}|f|d\mu,
$$
we obtain the inequality
$$
(1-\delta)^{d+1}\le (Cd)^{2d}
\biggl(2s(d)+e4^d\bigl(\tau\delta+d\int_{\tau^{1/d}}\lambda^{d-1}e^{-\lambda}d\lambda\bigr)\biggl).
$$
Now set $\tau=\delta^{-1/2}$. Then, for any $\delta$ close enough to $0$,
 the left-hand side of the inequality
is close to $1$ and the right-hand side is close to $2(Cd)^{2d}s(d)$.
So, if we take $s(d)<2^{-1}(Cd)^{-2d}$, we obtain that
$\delta$ cannot be arbitrary small. Thus, the desired constant
$c(d)$ exists and the lemma is proved.
\end{proof}

\begin{corollary}\label{c2.1} Let $\mu$ be a log-concave measure on $\mathbb{R}^n$.
Then there is a constant $c(d)\in(0, 1)$ depending only on the degree $d$
such that for every polynomial $f$
of degree $d$ with $\sigma_f>0$ one has
$$
1-c(d)\ge\mu(f<m_f)\ge c(d).
$$
\end{corollary}

\vskip .1in

{\bf A. The case of a Gaussian measure}

\vskip .1in

First we consider the case of the standard Gaussian measure on $\mathbb{R}^n$
and a polynomial of an arbitrary degree.

Below we need the following elementary estimate.
Let $f$ be a smooth function on $\mathbb{R}^n$.
Let $D^mf(x)$ denote the $m$-fold derivative of  $f$, i.e., a multilinear function such that
$$
\frac {d^m}{dt^m}f(x+th)|_{t=0}=D^mf(x)(h,\ldots,h).
$$
Let now $f$ be a polynomial of degree $d$ on $\mathbb{R}^n$. Set
$$
\varphi(t)=f(x+t(y-x)).
$$
Then one has
$$
f(y)=\varphi(1)=\varphi(0)+\sum_{m=1}^d(m!)^{-1}\varphi^{(m)}(0)
=f(x)+\sum_{m=1}^d(m!)^{-1}D^mf(x)(y-x,\ldots,y-x),
$$
thus,
$$
f(y)-f(x)\le\sum_{m=1}^d(m!)^{-1}|D^mf(x)||y-x|^m,
$$
where $|D^mf(x)|^2=\sum_{i_1,\ldots,i_m}|\partial_{x_{i_1},\ldots,x_{i_m}}f(x)|^2$.
Applying this estimate to the polynomial $-f$, we get
\begin{equation}\label{2.1}
|f(y)-f(x)|\le\sum_{m=1}^d(m!)^{-1}|D^mf(x)||y-x|^m.
\end{equation}

We also need the following estimate for the standard Gaussian measure $\gamma$ on $\mathbb{R}^n$:
\begin{equation}\label{2.2}
\|D^mf\|^2_2:=\int|D^mf|^2d\gamma\le a(m,d)\sigma_f^2,
\end{equation}
where the constant $a(m,d)$ depends only on $m$ and $d$ and is independent of the dimension~$n$.
This estimate follows from the equivalence of the Sobolev and
$L^p$-norms on the space of all polynomials of a fixed degree (see \cite{Gaus}).

\vskip .1in

Let us recall the isoperimetric inequality for Gaussian measures
(see \cite{STs, Bor3, Gaus}), which is used in the proof of the next theorem.
Let $\gamma$ be the standard Gaussian measure on $\mathbb{R}^n$,
let $C$ be a measurable set in $\mathbb{R}^n$, and let $B$ be
the closed unit ball centered at the origin. Then
\begin{equation}\label{2.3}
\gamma(C+\varepsilon B)\ge\Phi(a+\varepsilon),
\end{equation}
where $\Phi$ is the distribution function of the standard Gaussian random variable on the real line and
the number $a$ is chosen from the relation $\gamma(C)=\Phi(a)$.

For a Gaussian measure, there is a better estimate
for the tails of a polynomial $f$ of degree $d$ than in Theorem \ref{t1.1}:
\begin{equation} \label {2.4}
\gamma(|f|>\|f\|_2 t^d)\le R \exp\bigl(-rt^2\bigr)
\end{equation}
for every $t>0$ and some positive constants $R$ and $r$ (see \cite[Corollary 5.5.7]{Gaus}).

\begin{theorem}\label{t2.1}
Let $\gamma$ be the standard Gaussian measure on $\mathbb{R}^n$. Then there is a number
$L(d)>0$ depending only on the degree $d$ such that for every polynomial $f$ of
degree $d$ the following estimate holds:
$$
\gamma(|f-m_f|\le\sigma_f s)\ge L(d)s|\ln s|^{-d/2} \quad \text{for}\ 0\le s\le 1/2.
$$
\end{theorem}
\begin{proof}
Without loss of generality we can assume that $m_f=0$.
Fix a number $t>1$ (this number will be chosen later).
Let $B$ be the unit ball in $\mathbb{R}^n$ centered at the origin.
For $\varepsilon<1$ consider the set
$$
A=(\{f<0\}+\varepsilon B)\bigcap(\{f>0\}+\varepsilon B)\bigcap\bigcap_{m=1}^d\bigl\{|D^mf(x)|\le t^{d-m} \|D^mf\|_2\bigr\}.
$$
Note that by (\ref{2.1}) we have
$A\subset\bigl\{|f|\le\sum_{m=1}^d(m!)^{-1} t^{d-m}\|D^mf\|_2\varepsilon^m\bigr\}$.
Since $t>1$ and $\varepsilon<1$, we have
$\varepsilon^m\le\varepsilon$, $t^{d-m}\le t^d$, and inequality (\ref{2.2})
yields the estimate
$$
\sum_{m=1}^d(m!)^{-1} t^{d-m}\|D^mf\|_2\varepsilon^m\le
C(d)\sigma_ft^d\varepsilon,
$$
where $C(d)$ is a constant.
Thus, we have the inclusion
$$
A\subset\{|f|\le C(d)\sigma_ft^d\varepsilon\}.
$$
Let $a_f$ be the number such that $\gamma(f<0)=\Phi(a_f)$.
Using the Gaussian isoperimetric inequality (\ref{2.3}), we can estimate the
measure of the set $A$:
\begin{multline*}
\gamma(A)\ge 1-(1-\gamma(\{f<0\}+\varepsilon B))-(1-\gamma(\{f>0\}+\varepsilon B))\\
-\sum_{m=1}^d\gamma\{|D^mf(x)|\ge t^{d-m} \|D^mf\|_2 \}\\
\ge-1+\Phi(a_f+\varepsilon)+\Phi(-a_f+\varepsilon)-R \sum_{m=1}^de^{-r t^2}\\
=\Phi(a_f+\varepsilon)-\Phi(a_f)+\Phi(-a_f+\varepsilon)-\Phi(-a_f)-R de^{-r t^2}.
\end{multline*}
In the last inequality Theorem \ref{t1.1} is used.

By Corollary \ref{c2.1}, there is a number $c(d)$ such that
$$
1-c(d)\ge\gamma(f<m_f)\ge c(d),
$$
so, there is a constant $a(d)$ such that for every polynomial $f$ of degree $d$
with zero expectation one has $|a_f|\le a(d)$.
Let $q(d)$ be the minimal value of the standard Gaussian density on the interval $[-a(d)-1,a(d)+1]$.
Then
$\Phi(a+\varepsilon)-\Phi(a)\ge q(d)\varepsilon$ and similarly
$\Phi(-a+\varepsilon)-\Phi(-a)\ge q(d)\varepsilon$.
Therefore,
$$
\gamma(|f|\le C(d)\sigma_ft^d\varepsilon)\ge\gamma(A)\ge 2q(d)\varepsilon-R de^{-rt^2}.
$$
Let $0<s<1$, $\varepsilon=C(d)^{-1}2^{-d}r^{d/2}s|\ln s|^{-d/2}$, $t=2r^{-1/2}|\ln s|^{1/2}$. Then
\begin{multline*}
\gamma(|f|\le\sigma_fs)\ge 2q(d)C(d)^{-1}2^{-d}r^{d/2}s|\ln s|^{-d/2}-Rde^{-2|\ln s|}\\=
2q(d)C(d)^{-1}2^{-d}r^{d/2}s|\ln s|^{-d/2}-Rds^2.
\end{multline*}
There is a number $s_0(d)\in (0, e^{-1})$ such that for any
$s\le s_0(d)$ one has
$$
\gamma(|f|\le\sigma_fs)\ge q(d)C(d)^{-1}2^{-d}r^{d/2}s|\ln s|^{-d/2}.
$$
Thus, there is a positive constant $L(d)$ such that
$$
\gamma(|f|\le\sigma_fs)\ge L(d)s|\ln s|^{-d/2},
$$
whenever $0\le s\le 1/2$.
\end{proof}

\vskip .2in

{\bf B. The case of a log-concave measure and a polynomial of degree two}
\vskip .1in

Here we obtain an estimate that is sharper than the previous one
in the case of an arbitrary log-concave measure,
but applies only to polynomials of degree $2$.
The proof in this case relies on the following lemma.

\begin{lemma}\label{lem2.2}
Let $s\ge0$.
Then there is a constant $c$ such that the inequality
$$
\varepsilon\int_0^se^{-t}I_{\{f\le-\varepsilon\}}dt\int_0^se^{-t}I_{\{f\ge\varepsilon\}}dt
\le c\int_0^se^{-t}I_{\{|f|<\varepsilon\}}dt\int_0^se^{-t}|f(t)|dt
$$
holds for every polynomial $f$ of degree two on the real line.
\end{lemma}
\begin{proof}
Without loss of generality we can
assume that the coefficient at the highest degree term of the polynomial $f$ is $1$.
Suppose first that $s\ge1$.

If the right-hand side is not zero, then there is a root of the polynomial $f$ on the interval $(0, s)$.
Let $\tau$  denote the minimal root in $(0, s)$  and let $\sigma$ be the second root.
Then on the interval $(0, \tau)$ the polynomial $f$ is either strictly positive or strictly negative,
hence,
$$
\int_0^se^{-t}I_{\{f\le-\varepsilon\}}dt\int_0^se^{-t}I_{\{f\ge\varepsilon\}}dt.
\le\int_\tau^\infty e^{-t}dt=e^{-\tau}
$$
Thus, we have to obtain the estimate
$$
\varepsilon e^{-\tau}\le c\int_0^se^{-t}I_{\{|f|<\varepsilon\}}dt\int_0^se^{-t}|f(t)|dt.
$$

Let us consider the case $\varepsilon<\frac{1}{2}(1+|\tau-\sigma|)$.
By Theorems \ref{t1.2} and \ref{t1.3}, given a polynomial $g$ of degree $d$,
linear functions $\ell_1, \ell_2$ and
a log-concave measure $\nu$, we have
$$
\|g\|_{L^1(\nu)}\ge c^d \|g\|_{L^2(\nu)},
$$
$$
c_1\|\ell_1\|^2_{L^2(\nu)}\|\ell_2\|^2_{L^2(\nu)}\ge\|\ell_1\ell_2\|^2_{L^2(\nu)}\ge c_2\|\ell_1\|^2_{L^2(\nu)}\|\ell_2\|^2_{L^2(\nu)},
$$
$$
\int_{A}|g|d\nu\ge c(d) \nu(A)^{1+d}\int|g|d\nu.
$$
Applying this inequalities to the functions $g=f$, $\ell_1(t)=t-\tau$, $\ell_2(t)=t-\sigma$
and to the measure $e^{-t}I_{\{t>0\}}dt$, we obtain that
\begin{multline*}
\int_0^se^{-t}|f(t)|dt\ge C(1-e^{-s})^3\int_0^\infty e^{-t}|f(t)|dt\\
\ge\widetilde{C}(1-e^{-s})^3
\biggl(\int_0^\infty(t-\sigma)^2e^{-t}dt\int_0^\infty(t-\tau)^2e^{-t}dt\biggr)^{1/2}\\
\ge\widehat{C}\bigl(1+(1-\tau)^2\bigr)^{1/2}\bigl(1+(1-\sigma)^2\bigr)^{1/2}.
\end{multline*}
In the last estimate we have used the inequality $1-e^{-s}>1-e^{-1}$ for $s\ge1$.
Let us pick a point $u$ such that $\tau\in[u, u+1]\subset[0, s]$. Then
$|t-\sigma|\le|\tau-\sigma|+1$ on $[u, u+1]$. Thus,
$$
\int_0^se^{-t}I_{\{|f|<\varepsilon\}}dt\ge e^{-\tau-1}\int_u^{u+1}
I_{\{|t-\tau|<\varepsilon(|\tau-\sigma|+1)^{-1}\}}dt\ge  e^{-\tau-1}\varepsilon(|\tau-\sigma|+1)^{-1}.
$$
Since
$$
\inf_{\tau, \sigma}\frac{\bigl(1+(1-\tau)^2\bigr)^{1/2}
\bigl(1+(1-\sigma)^2\bigr)^{1/2}}{|\tau-\sigma|+1}\ge1/2>0,
$$
the desired estimate is proved in this case.

Now let us consider the case $\varepsilon\ge\frac{1}{2}(1+|\tau-\sigma|)$.
If
$$\varepsilon\le  2(1-e^{-1})^{-1}\int_0^\infty e^{-t}|f(t)|dt,$$
then
$$
\int_0^se^{-t}I_{\{|f|<\varepsilon\}}dt\ge e^{-\tau-1}\int_u^{u+1}I_{\{|t-\tau|<\varepsilon(|\tau-\sigma|+1)^{-1}\}}dt
\ge e^{-\tau-1}\int_u^{u+1}I_{\{|t-\tau|<1/2\}}dt\ge  e^{-\tau-1}/2.
$$
Using the estimate
$$
\int_0^se^{-t}|f(t)|dt\ge C \int_0^\infty e^{-t}|f(t)|dt,
$$
we obtain that the desired inequality is valid in this case too.
If
$$\varepsilon\ge  2(1-e^{-1})^{-1}\int_0^\infty e^{-t}|f(t)|dt,$$
then, by Chebyshev's inequality
$$
\int_0^se^{-t}I_{\{|f|\ge\varepsilon\}}dt
\le\int_0^\infty e^{-t}I_{\{|f|\ge\varepsilon\}}dt\le \varepsilon^{-1}\int_0^\infty e^{-t}|f(t)|dt.
$$
On the other hand,
$$
\int_0^se^{-t}I_{\{|f|<\varepsilon\}}dt\ge 1-e^{-1}-\varepsilon^{-1}\int_0^\infty e^{-t}|f(t)|dt\ge(1-e^{-1})2^{-1}.
$$
Using the above mentioned estimate
$$
\int_0^se^{-t}|f(t)|dt\ge C \int_0^\infty e^{-t}|f(t)|dt
$$
and estimating one of the integrals on the left-hand side by one,
we conclude that the desired estimate is also valid in the considered case.

Now let $s<1$. In this case $e^{-1}\le e^{-t}\le1$ on $[0, s]$ and our estimate
in this case is equivalent to the following one:
$$
\varepsilon\int_0^sI_{\{f\le-\varepsilon\}}dt\int_0^sI_{\{f\ge\varepsilon\}}dt
\le c\int_0^sI_{\{|f|<\varepsilon\}}dt\int_0^s|f(t)|dt.
$$
Moreover, after a linear change of variables,
we can assume $s=1$, i.e., the desired inequality takes the form
$$
\varepsilon\int_0^1I_{\{f\le-\varepsilon\}}dt\int_0^1I_{\{f\ge\varepsilon\}}dt
\le c\int_0^1I_{\{|f|<\varepsilon\}}dt\int_0^1|f(t)|dt.
$$
It can be easily seen that if we multiply the integrand by $e^{-t}$ in each integral,
then each integral
will differ from the former one by the factor which belongs to $[1,e]$.
So, this case follows from the previous one for $s=1$.
\end{proof}

\begin{remark}\label{rem2.1}{\rm
The estimate from Lemma \ref{lem2.2}
does not extend to polynomials of the third degree.
It is sufficient to take
$$
s=\infty, \quad f(t)=(t+1)^2(t-a).
$$
}
\end{remark}

\begin{remark}\label{rem2.2}{\rm
If the estimate from Lemma \ref{lem2.2}
were true for
$$
\int_0^se^{-t}|f(t)-r|dt\quad \text{instead of}\quad \int_0^se^{-t}|f(t)|dt
$$
for every $r>0$, the isoperimetric inequality in the Cheeger
form would be true for the distribution of any polynomial,
see Lemma \ref{lem3.1} and Theorem \ref{t3.1}.
However, this inequality cannot be true for an arbitrary log-concave measure.
Indeed, this inequality for an arbitrary polynomial of
degree two would have implied the exponential integrability of such polynomials
with respect to an arbitrary log-concave measure (see \cite{BobkDiss}).
}
\end{remark}

\begin{theorem}\label{t2.2}
There is a constant $c$ such that, for every polynomial
of the second degree on $\mathbb{R}^n$ and every log-concave measure $\mu$,
the following inequality holds:
$$
\varepsilon\int I_{\{f\le-\varepsilon\}}d\mu\int I_{\{f\ge\varepsilon\}}d\mu
\le c\int I_{\{|f|<\varepsilon\}}d\mu\int|f|d\mu.
$$
\end{theorem}

\begin{proof}
The functions $f_1=\varepsilon I_{\{f\le-\varepsilon\}}$, $f_2=I_{\{f\ge\varepsilon\}}$
are upper semi-continuous and
the functions $f_3=cI_{\{|f|<\varepsilon\}}$ and $f_4=|f|$ are lower semi-continuous.
Hence one can apply the localization lemma.
Thus, it is sufficient to prove the following inequality:
$$
\varepsilon\int_s^{r}e^{\ell(t)}I_{\{f\le-\varepsilon\}}dt\int_s^{r}e^{\ell(t)}I_{\{f\ge\varepsilon\}}dt
\le c\int_s^{r}e^{\ell(t)}I_{\{|f|<\varepsilon\}}dt\int_s^{r}e^{\ell(t)}|f|dt,
$$
where $\ell$ is an affine function.
By a linear change of variables this estimate can be reduced to the estimate
from Lemma \ref{lem2.2}.
Hence the theorem is proved.
\end{proof}

\begin{corollary}\label{c2.2}
There is a constant $C>0$ such that for every polynomial of degree $2$
on $\mathbb{R}^n$ and every log-concave measure $\mu$
the following estimate holds:
$$
\mu\{|f-m_f|<\varepsilon\}\int|f-m_f|d\mu\ge C\varepsilon \quad \text{for}\ \varepsilon<\alpha_f
$$
\end{corollary}

\begin{proof}
This estimate follows from the previous theorem and Lemma \ref{lem2.1}.
\end{proof}

\section{The isoperimetric and Poincar\'e inequalities}

The following lemma is an analog of Lemma \ref{lem2.2}
for polynomials of an arbitrary degree,
but for measures with a density of the form $t^n$
on some interval instead of $e^{-t}$.

\begin{lemma}\label{lem3.1}
Let $s\ge0$, $n\in\mathbb{N}$.
Then there is a constant $c(d,n)$ depending only on $d$ and $n$ such that for every polynomial $f$
of degree $d$ on the real line one has
$$
\varepsilon\int_s^{s+1}t^nI_{\{f\le-\varepsilon\}}dt\int_s^{s+1}t^nI_{\{f\ge\varepsilon\}}dt
\le c(d,n)\int_s^{s+1}t^nI_{\{|f|<\varepsilon\}}dt\int_s^{s+1}t^n|f(t)|dt.
$$
\end{lemma}

\begin{proof}
Without loss of generality we can
assume that the coefficient at the highest degree term of the polynomial $f$ is $1$,
i.e., $f$ is of the form
$$
f(t)=\prod_{i=1}^d(t-t_i).
$$
Let
$$
\mu_s(dt):=\biggl(\int_s^{s+1}t^ndt\biggr)^{-1}I_{[s,s+1]}t^ndt,$$
$$
m_s:=\int t\mu_s(dt), \quad \sigma^2_s:=\int t^2\mu_s(dt)-m_s^2.
$$
Let
$$
c_1(n)=\inf_s\sigma^2_s.
$$
Note that $c_1(n)>0$, since the limit at infinity of $\sigma^2_s$ is not zero.
First we consider the case $\varepsilon<2\int|f|d\mu_s$.
Let
$$
\tau:=\max\{t_i: t_i\in[s, s+1]\}.
$$
Then on the interval $(\tau, s+1)$ the polynomial $f$ is either strictly positive or strictly negative.
Hence,
$$
\int_s^{s+1}t^nI_{\{f\le-\varepsilon\}}dt\int_s^{s+1}t^nI_{\{f\ge\varepsilon\}}dt
\le\int_s^{s+1}t^ndt\int_s^\tau t^ndt.
$$
Thus, it is sufficient to prove the estimate
$$
\varepsilon\int_s^\tau t^ndt\le c(d,n)\int_s^{s+1}t^nI_{\{|f|<\varepsilon\}}dt\int|f|d\mu_s.
$$
Applying Theorem \ref{t1.1} to the polynomial $f$ and the measure $\mu_s$, we obtain
\begin{multline*}
\|f\|_{L^1(\mu_s)}^2\ge c^{-2d}\|f\|_{L^2(\mu_s)}^2\ge
c^{-2d}\|f^2\|_{L^0(\mu_s)}=c^{-2d}\prod_{i=1}^d\|(t-t_i)^2\|_{L^0(\mu_s)}\\
\ge c^{-2d}(2c)^{-2d}\prod_{i=1}^d\|(t-t_i)^2\|_{L^1(\mu_s)}=
c^{-2d}(2c)^{-2d}\prod_{i=1}^d(\sigma^2_s+|m_s-t_i|^2).
\end{multline*}
Let $t_j$ be a root of the polynomial $f$ such that $\alpha_j:=\inf\limits_{t\in[s, s+1]}|t_j-t|>1$. Then
$$|f(t)|\le (\alpha_j+1)\prod\limits_{i\ne j}|t-t_i|$$
on $[s, s+1]$, while
$\sigma^2_s+|m_s-t_j|^2\ge\alpha_j^2$.
Let $t_k$ be a root of the polynomial $f$ such that
$\alpha_k:=\inf\limits_{t\in[s, s+1]}|t_k-t|\le1$.
Then $|f(t)|\le 2\prod\limits_{i\ne k}|t-t_i|$ on $[s, s+1]$,
while
$\sigma^2_s+|m_s-t_k|^2\ge c_1(n)$.
Applying these estimates several times, we obtain that
$$
|f(t)|\le 2^d\Bigl(\prod\limits_{\alpha_j>1}(\alpha_j+1)\Bigr)|t-\tau|\le4^d\Bigl(\prod\limits_{\alpha_j>1}\alpha_j\Bigr)|t-\tau|
$$
on $[s, s+1]$, while
$$
\prod\limits_{i=1}^d(\sigma^2_s+|m_s-t_i|^2)\ge(c_1(n))^d\prod\limits_{\alpha_j>1}\alpha^2_j.
$$
Let $R:=\prod\limits_{\alpha_j>1}\alpha_j$. Then it is sufficient to prove the estimate
$$
\varepsilon\int_s^\tau t^ndt\le c(d,n)\int_s^{s+1}t^nI_{\{|t-\tau|<\varepsilon4^{-d}R^{-1}\}}dt\int|f|d\mu_s.
$$
If $\varepsilon4^{-d}R^{-1}>1/4$, then
$$
\varepsilon\int_s^\tau t^ndt\le2\int|f|d\mu_s(n+1)^{-1}((s+1)^{n+1}-s^{n+1}),
$$
$$
\int_s^{s+1}t^nI_{\{|t-\tau|<\varepsilon4^{-d}R^{-1}\}}dt
\int|f|d\mu_s\ge(n+1)^{-1}((s+1/4)^{n+1}-s^{n+1})\int|f|d\mu_s.
$$
Obviously, $\inf\limits_s((s+1/4)^{n+1}-s^{n+1})((s+1)^{n+1}-s^{n+1})^{-1}>0$, thus,
in the case where $\varepsilon4^{-d}R^{-1}>1/4$ the desired inequality is proved.

Let now $\varepsilon4^{-d}R^{-1}\le1/4$. It is sufficient to prove the inequality
$$
\varepsilon\int_s^\tau t^ndt\le c(d,n)R\int_s^{s+1}t^nI_{\{|t-\tau|<\varepsilon4^{-d}R^{-1}\}}dt,
$$
since $\int|f|d\mu_s\ge \delta(d,n)R$. If $\tau<s+3/4$, then
$$
\int_s^{s+1}t^nI_{\{|t-\tau|<\varepsilon4^{-d}R^{-1}\}}dt\ge
\tau^n\varepsilon4^{-d}R^{-1},
$$
since $[\tau,\tau+\varepsilon 4^{-d} R^{-1}]\subset[s,s+1]$,
while
$$
\int_s^\tau t^ndt\le \tau^n (\tau-s) \le \tau^n.
$$
Thus, in this case the desired inequality is also proved.
Let now $\tau\ge s+3/4$. Then
$$
\int_s^{s+1}t^nI_{\{|t-\tau|<\varepsilon4^{-d}R^{-1}\}}dt\ge
(\tau-\varepsilon4^{-d}R^{-1})^n\varepsilon4^{-d}R^{-1},
$$
since $[\tau,\tau-\varepsilon 4^{-d} R^{-1}]\subset[s,s+1]$.
The last expression can be estimated by
$$
(\tau-1/4)^n\varepsilon4^{-d}R^{-1}\ge(s+1/2)^n\varepsilon4^{-d}R^{-1}.
$$
Moreover,
$$
\int_s^\tau t^ndt=(n+1)^{-1}(\tau^{n+1}-s^{n+1})\le(s+1)^n
$$
and, since $s+1/2\ge 1/2(s+1)$, the desired inequality is proved in this case too.

It remains to consider the case where $\varepsilon\ge 2\int|f|d\mu_s$.
Applying Chebyshev's inequality, we obtain
$$
\mu_s(|f|\ge\varepsilon)\le \varepsilon^{-1}\int|f|d\mu_s\le 1/2,
$$
hence
$$
\mu_s(|f|<\varepsilon)\ge1-1/2=1/2.
$$
Thus, we have
$$
\varepsilon\mu_s\{f\le-\varepsilon\}\mu_s\{f\ge\varepsilon\}
\le c(d,n)\mu_s\{|f|<\varepsilon\}\int|f|d\mu_s,
$$
which is equivalent to the announced estimate.
\end{proof}

\begin{lemma}[The reverse Poincar\'e inequality]\label{lem3.2}
Let $f$ be a polynomial of degree $d$ on the real line and let $\mu$
be a log-concave measure on the real line.
Then the following inequality holds:
$$
\sigma(\mu)\|f'\|_{L^2(\mu)}\le (Cd)^d\|f\|_{L^2(\mu)},
$$
where $C$ is an absolute constant, $\sigma^2(\mu)$ is the variance of the measure $\mu$.
\end{lemma}
\begin{proof}
Due to homogeneity we can assume that the polynomial $f$ is of the form
$$
f(t)=\prod_{i=1}^d(t-t_i).
$$
Also, without loss of generality we can assume that $\int t\mu(dt)=0$.
Applying the estimate of the $L^1$-norm of a polynomial via its $L^0$-norm
from Theorem \ref{t1.2}, using multiplicativity of the
$L^0$-norm and estimating the $L^0$-norm by the $L^1$-norm, we obtain
\begin{multline*}
\int t^2\mu(dt)\int (f'(t))^2\mu(dt)\le d\sum_{i=1}^d\int t^2\mu(dt)
\int\Bigl|\prod_{j\ne i}(t-t_j)\Bigr|^2\mu(dt)\\
\le d(2cd)^{2d}\sum_{i=1}^d\int t^2\mu(dt)\prod_{j\ne i}\int|t-t_j|^2\mu(dt)\\
=d(2cd)^{2d}\sum_{i=1}^d\int t^2\mu(dt)\prod_{j\ne i}\biggl(\int t^2\mu(dt)+|t_j|^2\biggr)\\
\le d^2(2cd)^{2d}\prod_{i=1}^d\biggl(\int t^2\mu(dt)+|t_i|^2\biggr)=\\
d^2(2cd)^{2d}\prod_{i=1}^d\int|t-t_i|^2\mu(dt)
\le d^2(4c^2d)^{2d}\int f^2\mu(dt),
\end{multline*}
as announced.
\end{proof}

\begin{lemma}\label{lem3.3}
Let $f$ be a polynomial of degree $d$ on $\mathbb{R}$. Suppose that $f$ vanishes
at some point $\tau\in[s, s+1]$, where $s\ge0$. Then
$$
\int_s^{s+1}|f(t)|t^ndt\le C(d,n)\int_s^{s+1} |f(t)-r|t^ndt
$$
for every number $r$, where the constant $C(d,n)$ depends only on $d$ and $n$.
In particular, $C(d,n)$ is independent of $r$.
\end{lemma}

\begin{proof}
Let $\mu_s, \sigma_s$ be defined as in the proof of Lemma \ref{lem3.1}.
Note that it is sufficient to prove the inequality
$$
\int|f|d\mu_s\le C(d,n)\int|f-\lambda|d\mu_s.
$$
It can be easily verified that
$$
\|f\|_{L^\infty(\mu_s)}\le\|f'\|_{L^\infty(\mu_s)}.
$$
Recall the following inequality
(see Theorem 1 and Corollary after it in \cite{CarWr}) that
 holds
with some universal constant~$C$
 for any probability measure
on some interval with a density of the form $ct^n$ and for any polynomial $g$ of degree~$d$:
$$
\|g\|_\infty\le \biggl(C\max\biggl\{1,\frac{n+1}{d}\biggr\}\biggr)^{d}\|f'\|_2
$$
Applying this estimate to our measure $\mu_s$ and the polynomial $f'$, we obtain
$$
\|f'\|_{L^\infty(\mu_s)}\le \biggl(C\max\biggl\{1,\frac{n+1}{d-1}\biggr\}\biggr)^{d-1}\|f'\|_{L^2(\mu_s)}.
$$
We now use Lemma \ref{lem3.2}:
$$
\|f'\|_{L^2(\mu_s)}\le (Cd)^d\|f - \lambda\|_{L^2(\mu)}\sigma_s^{-1}.
$$
Since $\min\limits_s\sigma_s>0$ and $\|f\|_{L^1(\mu_s)}\le\|f\|_{L^\infty(\mu_s)}$, the lemma is proved.
\end{proof}

\begin{corollary}\label{c3.1}
Let $f$ be a polynomial of degree $d$ on the real line, $s\ge0$, $n\in\mathbb{N}$.
Then there is a constant $c(d,n)$ depending only on $d$ and $n$ such that for every number $r$
one has
$$
\varepsilon\int_s^{s+1}t^nI_{\{f\le-\varepsilon\}}dt\int_s^{s+1}t^nI_{\{f\ge\varepsilon\}}dt
\le c(d,n)\int_s^{s+1}t^nI_{\{|f|<\varepsilon\}}dt\int_s^{s+1}t^n|f(t) - r|dt.
$$
\end{corollary}

\begin{proof}
If $f$ has no zeros in $[s, s+1]$, then the left-hand side vanishes and the inequality is obvious.
If $f$ has a root in $[s, s+1]$, then we can apply Lemmas \ref{lem3.1} and \ref{lem3.3}.
\end{proof}

\begin{theorem}\label{t3.1}
Let $K$ be a convex compact set in $\mathbb{R}^n$, let $\lambda_K$ be the normalized Lebesgue measure on $K$,
and let $f$ be a polynomial of degree $d$.
Let $\mathbb{R}=J_1\sqcup J_2\sqcup J_3$, where $J_i$ are disjoint measurable sets,
and let the distance between $J_1$ and $J_3$
be $\varepsilon>0$, i.e.,
$$
    \inf\limits_{x_1 \in J_1, x_2 \in J_2}|x_1-x_2|=\varepsilon>0.
$$
Then there is a constant $c(d,n)$ depending only on $d$ and $n$ such that
$$
\varepsilon\lambda_K(f\in J_1)\lambda_K(f\in J_3)\le c(d,n)\lambda_K(f\in J_2)\int_K|f-m_f|d\lambda_K,
$$
where $\displaystyle \ m_f=\int f d\lambda_K$.
\end{theorem}

\begin{proof}
Obviously, it is sufficient to prove the estimate with an arbitrary number $m$ in place of $m_f$.
We will prove exactly this statement.

If we take the closures of the sets $J_1$ and $J_3$ and replace $J_2$ with
 $\mathbb{R}\setminus(J_1\cup J_3)$,
the left-hand side of the inequality does not decrease, while the right-hand side does not increase.
Hence we can assume that $J_1$ and $J_3$ are closed and $J_2$ is open.
First we consider the case where $J_2$ is an interval.
Note that our inequality can be written in the form
$$
\varepsilon\int_KI_{\{f\in J_1\}}(x)dx\int_KI_{\{f\in J_3\}}(x)dx\le c(d,n)
\int_KI_{\{f\in J_2\}}(x)dx\int_K|f(x)-m|dx.
$$
Since $J_1$ and $J_3$ are closed and $J_2$ is open, the
functions $f_1=\varepsilon I_{\{f\in J_1\}}$, $f_2=I_{\{f\in J_3\}}$ are upper semi-continuous and
the functions $f_3=c(d,n)I_{\{f\in J_2\}}$, $f_4=|f-\lambda|$ are lower semi-continuous.
Hence we can apply Theorem \ref{t1.5}.
Thus, it is sufficient to prove the inequality
$$
\varepsilon\int_s^{r}\ell(t)^{n-1}I_{\{f\in J_1\}}dt\int_s^{r}\ell(t)^{n-1}I_{\{f\in J_3\}}dt
\le c(d,n)\int_s^{r}\ell(t)^{n-1}I_{\{f\in J_2\}}dt\int_s^{r}\ell(t)^{n-1}|f-\lambda|dt,
$$
where $\ell$ is an affine function that is nonnegative on $[s, r]$.
By a linear change of variables we arrive at the inequality
$$
\varepsilon\int_s^{s+1}t^{n-1}I_{\{f\in J_1\}}dt\int_s^{s+1}t^{n-1}I_{\{f\in J_3\}}dt
\le c(d,n)\int_s^{s+1}t^{n-1}I_{\{f\in J_2\}}dt\int_s^{s+1}t^{n-1}|f-\lambda|dt,
$$
where $s\ge0$ and $f$ is some (possibly, different) polynomial.
Thus, the case where $J_2$ is an open interval follows from Corollary \ref{c3.1}.

Recall that
$\mu_f=\lambda_K\circ f^{-1}$ is the image of the uniform distribution
on the compact set $K$ under the polynomial mapping $f$.
Note that the support of the measure $\mu_f$ is a bounded set on the real line.
Let now $J_2$ be the union of countably many disjoint intervals.
The proof in this case is similar to the one from \cite{KLS}.
For completeness, we present this argument.
If there is an
interval which length
is less than $\varepsilon$, then its both end points belong either to
$J_1$ or to $J_3$ and we can add this interval to $J_1$ or $J_3$, respectively, and prove
the estimate for exactly these sets. Thus, we assume that
$J_2=\bigsqcup\limits_i(a_i,b_i)$ (the union of disjoint intervals), moreover,
$b_i-a_i\ge\varepsilon$ and there are finitely many such intervals,
since the support of the measure $\mu_f$ is bounded.
We have already proved that
$$
\varepsilon\mu_f\bigl((-\infty, a_i]\bigr)\mu_f\bigl([b_i, \infty)\bigr)\le c(d,n) \mu_f\bigl((a_i, b_i)\bigr)\alpha_f.
$$
Summing these inequalities in $i$, we obtain
$$
\sum_i\mu_f\bigl((-\infty, a_i]\bigr)\mu_f\bigl([b_i, \infty)\bigr)
\le \frac{c(d,n)\alpha_f}{\varepsilon}\mu_f(J_2).
$$
Since every point of $J_1$ and every point of $J_3$ are separated by at least one interval $(a_i, b_i)$,
we have
$$
\sum_i\mu_f\bigl((-\infty, a_i]\bigr)\mu_f\bigl([b_i, \infty)\bigr)\ge \mu_f(J_1)\mu_f(J_3),
$$
which completes the proof.
\end{proof}

\begin{corollary}[The isoperimetric inequality in Cheeger's form]\label{c3.2}
Let $K$ be a convex compact set in $\mathbb{R}^n$, let $\lambda_K$ be the normalized Lebesgue measure on $K$,
and let $f$ be a polynomial of degree $d$.
Then there is a constant $\delta(d,n)$ depending only on $d$ and $n$ such that
$$
\mu_f^+(A)\ge \frac{\delta(d,n)}{\alpha_f}\mu_f(A)\mu_f(\mathbb{R}\setminus A).
$$
\end{corollary}
\begin{proof}
It is sufficient to apply the previous theorem to the sets
$$J_1 = A, \ J_3= \mathbb{R} \setminus\bigl( A + (-\varepsilon, \varepsilon) \bigr), \
J_2 = \mathbb{R} \setminus (J_1 \cup J_2)$$
and let $\varepsilon$ tend to zero.
\end{proof}

As it is known (see \cite{Cheeg},
a proof can also be found in \cite{BobkDiss}),
the previous assertion implies the following result.

\begin{corollary}[The Poincar\'e inequality]\label{c3.3}
Let $K$ be a convex compact set in $\mathbb{R}^n$, let $\lambda_K$ be the normalized Lebesgue measure on $K$,
and let $f$ be a polynomial of degree $d$.
Then there is a constant $C(d,n)$ depending only on $d$ and $n$ such that,
for every smooth function $\varphi$, one has
$$
\biggl\|\varphi-\int\varphi d\mu_f\biggr\|_{L^2(\mu_f)}\le C(d,n) \alpha_f\|\varphi'\|_{L^2(\mu_f)}.
$$
\end{corollary}

\end{document}